\documentclass[11pt]{amsart}

\usepackage{amscd,amssymb,amsopn,amsmath,amsthm,mathrsfs,graphics,amsfonts,enumerate,verbatim,calc
}
\usepackage{bbm}
\usepackage[all,cmtip]{xy}
\usepackage[all]{xy}
\usepackage{tikz}
\usepackage{lscape}
\usepackage{enumitem}
\usetikzlibrary{matrix,arrows,decorations.pathmorphing,cd}
\usepackage{hyperref}
\hypersetup{colorlinks=true,linkcolor={blue}}
\usepackage{cleveref}
\usepackage{url}

\usepackage{xurl}

\usepackage{scalerel}
\usepackage{stackengine,wasysym}
\usetikzlibrary {positioning}
\usetikzlibrary{patterns}
\usetikzlibrary{calc}
\definecolor {processblue}{cmyk}{0.96,0,0,0}

\usepackage{subfigure}

\usepackage{comment} 

\usepackage{ dsfont }

\usepackage{color}

\usepackage[OT2,OT1]{fontenc}
\newcommand\cyr{%
\renewcommand\rmdefault{wncyr}%
\renewcommand\sfdefault{wncyss}%
\renewcommand\encodingdefault{OT2}%
\normalfont
\selectfont}
\DeclareTextFontCommand{\textcyr}{\cyr}

\usepackage{amssymb,amsmath}

\DeclareFontFamily{OT1}{rsfs}{}
\DeclareFontShape{OT1}{rsfs}{n}{it}{<-> rsfs10}{}
\DeclareMathAlphabet{\mathscr}{OT1}{rsfs}{n}{it}

\numberwithin{equation}{section}
\newtheorem*{arc}{The Auslander-Reiten Conjecture}
\newtheorem{theorem}{Theorem}[section]

\theoremstyle{definition}
\newtheorem{defn}[theorem]{Definition}
\theoremstyle{remark}
\newtheorem{remark}[theorem]{Remark}

\def\tac{(\mathbf{tac})}
\def\tr{(\mathbf{tr})}

\newcommand{\Gdim}{\operatorname{G-dim}}

\newcommand{\Ass}{\operatorname{Ass}}

\renewcommand{\ker}{\operatorname{ker}}

\newcommand{\RHom}{\operatorname{RHom}}

\newcommand{\Ext}{\operatorname{Ext}}

\newcommand{\Tor}{\operatorname{Tor}}
\newcommand{\Hom}{\operatorname{Hom}}

\newcommand{\coker}{\operatorname{coker}}

\newcommand{\Tr}{\operatorname{Tr}}

\newcommand{\p}{\mathfrak{p}}

\newcommand{\m}{\mathfrak{m}}

\begin{document}
\title[On The Totally Acyclic Condition and the Auslander-Reiten Conjecture]{On The Totally Acyclic Condition and the Auslander-Reiten Conjecture}

\author[Lyle]{Justin Lyle}
\email[Justin Lyle]{jll0107@auburn.edu}
\urladdr{https://jlyle42.github.io/justinlyle/}
\address{Department of Mathematics and Statistics \\ 305 W Samford Avenue \\
Auburn University, AL 36849}



\begin{abstract}
We prove new cases of the Auslander-Reiten conjecture under the assumption that the base ring satisfies the totally acyclic condition $(\mathbf{tac})$. A consequence of our work is that Auslander-Reiten conjecture holds in the graded setting whenever $R$ is Gorenstein and reduced. We also give an example of a local ring $R$ which does not satisfy $(\mathbf{tac})$ but for which the condition $\operatorname{Ext}^{i>0}_R(M,R)=0$ implies $M$ is totally reflexive for any finitely generated $R$-module $M$. 
\end{abstract}

\maketitle

\section{Introduction}

Let $R$ be a commutative Noetherian ring and let $M$ be a finitely generated $R$-module. The famous Auslander-Reiten conjecture states the following:
\begin{arc}\label{arc}
Suppose $M$ is a finitely generated $R$-module for which $\Ext^{i>0}_R(M,M \oplus R)=0$. Then $M$ is projective.
\end{arc}

The Auslander-Reiten conjecture is a porting to commutative algebra of the analogous version introduced by Auslander-Reiten to attack several fundamental problems that remain open in the representation theory of Artin algebras \cite{AR75}. The Auslander-Reiten conjecture has received a great deal of attention in commutative algebra, but especially in its motivating case, that of Artinian algebras, it remains widely open.

In this work we prove new cases of the Auslander-Reiten conjecture under the assumption that $R$ satisfies the totally acyclic condition $\tac$: We say $R$ satisfies $\tac$ if for every acyclic complex $X$ of finitely generated projective $R$-modules, we have that $\Hom_R(X,R)$ is also acyclic (see \Cref{tacdef}). Every Gorenstein ring satisfies $\tac$ \cite[Corollary 5.5]{IK06}, but the $\tac$ condition is more broad, as for instance every local ring $(R,\m,k)$ with $\m^2=0$ satisfies $\tac$, as do other rings with various small numerics \cite[Corollary 1.1]{HJ09}.

Our main result in this direction is the following:

\begin{theorem}\label{thm1intro}
Suppose $R$ is a positively graded algebra over a field $k$ and let $M$ be a finitely generated graded $R$-module. If $R$ satisfies $\tac$ and $M$ is generically free, then $\Ext^{i>0}_R(M,M \oplus R)=0$ implies that $M$ is free.
\end{theorem}

As an immediate consequence of \Cref{thm1intro}, we have the following:
\begin{theorem}\label{gorconsequence}
Suppose $R$ is a positively graded algebra over a field $k$ that is Gorenstein and reduced. If $M$ is a finitely generated graded $R$-module with $\Ext^{i>0}_R(M,M \oplus R)=0$, then $M$ is free. That is, the Auslander-Reiten conjecture holds in this setting.
\end{theorem}

Recently, examples were found of positively graded Gorenstein domains $R$ over a field $k$ and homogeneous nonprincipal ideals $I$ in $R$ for which $\Ext^1_R(I,I)=0$, and which thus give counterexamples to the Huneke-Wiegand Conjecture, an extreme strengthening of the Auslander-Reiten Conjecture for such rings \cite{CI26,Ph26}. Theorem \ref{gorconsequence} shows as a byproduct that such examples cannot be upgraded to provide a counterexample to the Auslander-Reiten Conjecture, as the conjecture holds true in their setting.

A condition related to the $\tac$ condition is the totally reflexive condition $\tr$: We say $R$ say satisfies $\tr$ if for any finitely generated $R$-module $M$, the condition that $\Ext^{i>0}_R(M,R)=0$ implies that $M$ is totally reflexive (see \Cref{gdimdef} for the definition). We give a new class of rings satisfying $\tr$:

\begin{theorem}\label{thm2intro}
Suppose $(R,\m,k)$ is a local ring with $\m^3=0$ and suppose that $k$ is finite. Then $R$ satisfies $\tr$.
\end{theorem}

The condition $\tac$ always implies $\tr$ (see e.g. \cite[Theorem 1.4]{KL23}), but it has been an open question as to whether the converse holds. Using Theorem \ref{thm2intro}, we appeal to a construction of Hughes-Jorgensen-\c{S}ega to give an example of a local $R$ which satisfies $\tr$ but fails to satisfy $\tac$ (see Theorem \ref{tactrcountex}).

We now outline the structure of this work. Section \ref{sectionprelim} sets notation and conventions, and gives the necessary background for this work. Section \ref{sectionar} contains our results on the Auslander-Reiten conjecture, proving Theorems \ref{thm1intro} and \ref{gorconsequence}, while Section \ref{sectiontac} contains the proof of Theorem \ref{thm2intro}.

\section{Preliminaries}\label{sectionprelim}

Throughout, we let $(R,\m,k)$ be commutative Noetherian ring which is either local or positively graded over the field $k$ with homogeneous maximal ideal $\m$. All modules are assumed to be finitely generated and also graded in the graded case. If $M$ is an $R$-module, we let $\Omega^i_R(M)$ denote the $i$th syzygy of $M$.

We write $(-)^*:=\Hom_R(-,R)$, and if $A$ is a minimal presentation matrix (homogeneous in the graded setting) for $M$, we let $\Tr_R(M):=\coker(A^T)$ denote the \emph{Auslander transpose} of $M$. See \cite{AB69} or \cite[Chapter 12]{LW12} as a references for the standard properties of $\Tr_R(M)$. Of particular relevance for our needs, we have from e.g. \cite[Remark 12.2]{LW12}) that, for any pair of $R$-modules $M,N$, there is a natural map $\epsilon_{M,N}:M^* \otimes_R N \to \Hom_R(M,N)$ given by $f \otimes x \mapsto \alpha_{f,x}$ where $\alpha_{f,x}(y)=f(y)x$, and which fits into the so-called Auslander-Bridger exact sequence
\begin{equation}\label{auslanderbridgerexact}
0 \rightarrow \Tor^R_2(\Tr_R(M),N) \rightarrow M^* \otimes_R N \xrightarrow{\epsilon_{M,N}} \Hom_R(M,N) \to \Tor^R_1(\Tr_R(M),N) \to 0.
\end{equation}
It is easy to see that $\epsilon_{M,M}$ is a homogeneous map of degree $0$ in the graded setting, and it is well-known that $\epsilon_{M,M}$ is surjective if and only if $M$ is free (see e.g. \cite[Remark 12.2]{LW12}).

Next we recall the notion of Gorenstein dimension:
\begin{defn}\label{gdimdef}
We recall that an $R$-module $M$ is said to be \emph{totally reflexive} if $\Ext^{i>0}_R(M,R)=\Ext^{i>0}_R(\Tr_R(M),R)=0$. The \emph{Gorenstein dimension} of $M$ is defined as the value
\[\Gdim(M):=\inf\{i \mid \Omega^i_R(M) \text{ is totally reflexive}\}.\]
\end{defn}

The following are two key notions for this work:

\begin{defn}\label{tacdef}
We say $R$ satisfies $\tac$ (short for \emph{totally acyclic condition}) if for any acyclic complex $X$ of finitely generated projective $R$-modules, we have that $X^*$ is acyclic as well.   
\end{defn}

\begin{defn}\label{trdef}
We say $R$ satisfies $\tr$ if for any $R$-module $M$ with $\Ext^{i>0}_R(M,R)=0$, we have that $M$ is totally reflexive. 
\end{defn}

\begin{remark}
Any Gorenstein ring automatically satisfies $\tac$. In fact, it was shown in \cite[Corollary 5.5]{IK06} that $R$ is Gorenstein if and only if $X^*$ is acyclic for any complex $X$ of projective modules, with the caveat being that one must allow arbitrary projectives modules, not necessarily finitely generated for the backwards implication to hold.
\end{remark}

\begin{remark}
In \cite{Yo22} it is claimed that any generically Gorenstein ring satisfies $\tac$, however, this was later shown to be incorrect, with counterexamples given by \cite{KL23}. It remains unknown however, whether there is an excellent generically Gorenstein ring which fails to satisfy $\tac$, and more generally, the extent to which non-Gorenstein rings that satisfy $\tac$ are either abundant or scarce remains unknown.

\end{remark}

\section{Auslander-Reiten Conjecture for $\tac$ rings}\label{sectionar}
In this section, we leverage the $\tac$ condition to establish new cases of the Auslander-Reiten conjecture, proving \Cref{thm1intro} and \Cref{gorconsequence} from the intro.

Before proceeding with the proof of \Cref{thm1intro}, we show the $\tac$ condition forces $M^*$ and $M$ to interact well homologically under the cohomological vanishing condition $\Ext^{i>0}_R(M,M \oplus R)=0$. Before this, we recall that an $R$-module $M$ is said to be \emph{generically free} if $M_{\p}$ is a free $R_{\p}$-module for all $\p \in \Ass(R)$.  We also recall that for a complex of $R$-modules
\[X= \cdots \rightarrow X_{n+1} \xrightarrow{\partial^X_{n+1}} X_n \xrightarrow{\partial^X_n} X_{n-1} \rightarrow \cdots,\]
we have the corresponding \emph{soft truncations}:
\[X_{\ge i}:= \cdots \rightarrow X_{i+1} \rightarrow \ker(\partial^X_{i}) \rightarrow 0\]
and
\[X_{\le i}:= 0 \rightarrow \coker(\partial^X_{i+1}) \rightarrow X_{i-1} \rightarrow \cdots\]
at a value $i$.

\begin{theorem}\label{tacresult1}
Suppose $R$ satisfies $\tac$. If $M$ is an $R$-module with $\Ext^i_R(M,M \oplus R)=0$, then $M^* \otimes_R M$ is totally reflexive and $\Tor^R_i(M^*,M)=0$ for all $i>0$. If additionally $M$ is generically free, then $\Tor^R_2(\Tr(M),M)=0$, i.e., the map $\epsilon_M:M^* \otimes_R M \to \Hom_R(M,M)$ is injective. 
\end{theorem}

\begin{proof}
Since $R$ satisfies $\tac$, it follows from \cite[Theorem 1.4]{KL23} that it satisfies $\tr$. Then since $\Ext^i_R(M,R)=0$ for all $i>0$, we have that $M$ is totally reflexive. Let $F$ be a minimal free resolution of $M$ and $G$ a minimal free resolution of $M^*$. Since $M$ is totally reflexive, there is a quasiisomorphism $M \to G^*$. As $F$ is a bounded below complex of free $R$-modules, we thus have 
\[\RHom_R(M,M) \simeq \RHom_R(F,M) \simeq \RHom_R(F,G^*)=\Hom_R(F,G^*).\] 
But as $\Ext^i_R(M,M)=0$ for all $i>0$, it follows that \[H^i(\Hom_R(F,G^*)) \cong \begin{cases} \Hom_R(M,M) & i=0 \\ 0 & i>0 \end{cases}\]
Let $Y$ be the augmented complex 
\[0 \to \Hom_R(M,M) \to \Hom_R(F,G^*)_0 \to \Hom_R(F,G^*)_{-1} \to \cdots \]
with $\Hom_R(M,M)$ in degree $0$, and splice $Y$ with a minimal free resolution of $\Hom_R(M,M)$ to form $X$, which is an acyclic complex of finitely generated free $R$-modules. Since $R$ satisfies $\tac$, it follows that $X^*$ is acyclic as well. As $X^*$ is acyclic, we immediately obtain that $\Ext^i_R(\Hom_R(M,M),R)=0$ for all $i>0$. As $R$ satisfies $\tr$, it follows that $\Hom_R(M,M)$ is totally reflexive, and we note that $(X^*)_{\ge 0} \cong \Hom_R(F,G^*)^* \cong (F \otimes_R G)^{**} \cong F \otimes_R G$. It follows that $M^* \otimes_R M \cong H_0(F \otimes_R G) \cong \Hom_R(M,M)^*$. Thus $M^* \otimes_R M$ is totally reflexive, and we also see that $F \otimes_R G$ is a free resolution of $M^* \otimes_R M$; that is, that $\Tor^R_i(M^*,M)=0$ for all $i>0$. 

Now, if additionally $M$ is generically free, then the map $(\epsilon_M)_{\p}$ is an isomorphism for all $\p \in \Ass(R)$. But as $M^* \otimes_R M$ is totally reflexive, it in particular embeds in a nonzero free $R$-module $P$. If $\ker \epsilon_M \ne 0$, then we may take a minimal prime $\p$ of $\ker \epsilon_M$ which is necessarily in $\Ass_R(\ker \epsilon_M) \subseteq \Ass_R(P)=\Ass(R)$. But $\ker \epsilon_M$ is not supported on any associated prime of $R$, and it follows that $\ker \epsilon_M=0$. Finally, we note from \ref{auslanderbridgerexact} that $\ker \epsilon_M \cong \Tor^R_2(\Tr_R(M),M)$, completing the proof. 
\end{proof}

\begin{theorem}\label{tacmainthm}
Suppose $R$ is a positively graded algebra over a field $k$ which satisfies $\tac$. If $M$ is a generically free $R$-module for which $\Ext^i_R(M,M \oplus R)=0$, then $M$ is free. 
\end{theorem}

\begin{proof}
It follows from \Cref{tacresult1} that $\Tor^R_i(M^*,M)=0$ for all $i>0$ and that the map $\epsilon_M:M^* \otimes_R M \to \Hom_R(M,M)$ is injective. Since $\Ext^{i>0}_R(M,M)=\Ext^{i>0}_R(M,R)=0$, it follows from \cite[Theorem 1]{AB97} that $H_{M^*}(t)=\dfrac{H_M(t^{-1})H_R(t)}{H_R(t^{-1})}$ and that $H_{\Hom_R(M,M)}(t)=\dfrac{H_M(t^{-1})H_M(t)}{H_R(t^{-1})}$. Then from \cite[Lemma 7]{AB93}, we have 
\[H_{M^* \otimes_R M}(t)=\dfrac{H_{M^*}(t)H_M(t)}{H_R(t)}=\dfrac{H_M(t^{-1})H_R(t)H_M(t)}{H_R(t^{-1})H_R(t)}=\dfrac{H_M(t^{-1})H_M(t)}{H_R(t^{-1})}=H_{\Hom_R(M,M)}(t).\]
But then $\epsilon_M$ is a degree $0$ injection between two graded modules with the same Hilbert series, and is therefore an isomorphism. But this means $\Tor^R_1(\Tr_R(M),M)=0$, and it is well-known that this holds if and only if $M$ is free (see e.g. \cite[Remark 12.2]{LW12}), completing the proof. 
\end{proof}

\begin{theorem}\label{gormainthm}
Suppose $R$ is Gorenstein, reduced, and positively graded over the field $k$. If $M$ is an $R$-module with $\Ext^{i>0}_R(M,M \oplus R)=0$, then $M$ is free. 
\end{theorem}

\begin{proof}
Every Gorenstein ring satisfies $\tac$ by e.g. \cite[Corollary 5.5]{IK06} while every module $M$ over a reduced algebra is generically free, so the claim follows at once from \Cref{tacmainthm}.
\end{proof}

\begin{remark}
Introduced in 1994 as part of a larger focus on understanding torsion in tensor products, Huneke-Wiegand conjectured that for any Gorenstein domain $R$ of dimension $1$ and any $R$-module $M$, the condition that $\Ext^1_R(M,M)=0$ should force $M$ to be free. As we automatically have that $\Ext^i_R(M,R)=0$ for all $i>0$ in this setting, this conjecture could be thought of as an extreme strengthening of the Auslander-Reiten conjecture. Recently counterexamples were found using AI of positively graded Gorenstein domains $R$ over a field $k$ and homogeneous nonprincipal ideals $I$ in $R$ for which $\Ext^1_R(I,I)=0$ \cite{CI26,Ph26}. The authors in \cite{CI26} note explicitly that $\Ext^2_R(I,I) \ne 0$ in their examples, and so they do not give a counterexample to the Auslander-Reiten conjecture. But in fact, it follows from \Cref{gormainthm} that no counterexample can exist in this setting.
\end{remark}

\section{$\tac$ versus $\tr$}\label{sectiontac}

In this section we study the relationship between the $\tac$ and $\tr$ conditions. It is easy to see that $\tac$ implies $\tr$ (see e.g. \cite[Theorem 1.4]{KL23}), but using a construction of Hughes-Jorgensen-\c{S}ega \cite[2.4]{HJ09} we give an example showing the converse need not hold. First we provide one circumstance where $R$ will always satisfy $\tr$:

\begin{theorem}\label{mcubedfinitefield}
Suppose $\m^3=0$ and that $k$ is finite. Then $R$ satisfies $\tr$.
\end{theorem}

\begin{proof}
Suppose $M$ is an $R$-module with $\Ext^{i>0}_R(M,R)=0$. From \cite[1.2.7]{Ch00}, it suffices to show that $M$ has finite Gorenstein dimension, i.e., that some syzygy of $M$ is totally reflexive. By Matlis duality (\cite[Theorem 10.66]{Ro09}) we see that $\Tor^R_{i>0}(M,w)=0$. Now since $\m^3=0$ and since $\Omega^1_R(M) \hookrightarrow \m R^{\oplus \mu_R(M)}$, we may replace $M$ by $\Omega^1_R(M)$ to suppose that $\m^2 M=0$. It follows from \cite[Proposition 2.8]{HS04} that $M$ has constant Betti numbers. As $k$ is finite, and as $R$ has finite length, it follows that $R$ is finite as well. Then as there are only finitely many matrices over $R$ of size $\mu_R(M)$, it must be that $M$ has an eventually periodic minimal free resolution $F_M$, so we may replace $M$ be a $\Omega^i_R(M)$ for some $i \gg 0$ to suppose that $M$ is periodic. In particular, there is an exact complex $X$ of free $R$-modules with $\coker(\partial^X_0)=M$. Then $X^*$ is also a periodic complex, and we have $H^i(X^*)=0$ for $i>0$, since $\Ext^{i>0}_R(M,R)=0$. Therefore $X^*$ is acyclic, so $M$ is totally reflexive.
\end{proof}

\begin{remark}
We note that well-known examples constructed by Jorgensen-\c{S}ega show that the conclusion of \Cref{mcubedfinitefield} is false without the assumption that $k$ is finite \cite{JS06}. 
\end{remark}

\begin{theorem}\label{tactrcountex}
Let $R=k[x_1,x_2,x_3,x_4,x_5]/I$ where $I$ is the ideal
\[I=(x_1^2,x_4^2,x_2x_3,x_1x_2+x_2x_4,x_1x_3+x_3x_4,x_2^2,x_2x_5-x_1x_3,x_3^2-x_1x_5,x_4x_5,x_5^2,x_3x_5)\] and where $k$ is a finite field. Then $R$ satisfies $\tr$ but does not satisfy $\tac$.
\end{theorem}

\begin{proof}
It follows from taking $\alpha=1$ in \cite[2.4]{HJ09} that the complex 
\[X:\cdots \rightarrow R^{\oplus 2} \xrightarrow{\begin{pmatrix} x_1 & x_2 \\ x_3 & x_4 \end{pmatrix}} R^{\oplus 2} \xrightarrow{\begin{pmatrix} x_1 & x_2 \\ x_3 & x_4 \end{pmatrix}} R^{\oplus 2} \rightarrow \cdots\]
is acyclic but that $H_i(X^*) \ne 0$ for all $i$, so $R$ does not satisfy $\tac$. Moreover, the same result shows that $(R,\m,k)$ is a standard graded Artinian ring with Hilbert series $H_R(t)=1+5t+4t^2$. In particular, we have $\m^3=0$, and it follows from Lemma \ref{mcubedfinitefield} that $R$ satisfies $\tr$.
\end{proof}

\section*{Acknowledgements}

This material is based upon work supported by the National Science Foundation under Grant No. DMS-1928930 and by the Alfred P. Sloan Foundation under grant G-2021-16778, while the author was in residence at the Simons Laufer Mathematical Sciences Institute (formerly MSRI) in Berkeley, California, during the Spring 2024 semester.

\bibliographystyle{amsalpha}
\bibliography{mybib}

\vspace{.3cm}

\end{document}